\documentclass[11pt]{amsart}

\usepackage{latexsym, amsmath, amsthm, amssymb, setspace, verbatim, graphicx, amscd}
\usepackage[all]{xy}
\usepackage{longtable}
\usepackage{array}
\usepackage{hyperref}
\usepackage[ansinew]{inputenc}

\title{ The $K$-theory of the Compact Quantum Group $SU_{q}(2)$ for $q=-1$ }
\author{ Sel\c{c}uk Barlak }
\address{Westf\"alische Wilhelms-Universit\"at, Fachbereich Mathematik, \phantom{--------------}\linebreak \text{}\hspace{3.5mm} Einsteinstrasse 62, 48149 M\"unster, Germany}
\email{selcuk.barlak@uni-muenster.de}
\thanks{\emph{Supported by:} SFB 878 \emph{Groups, Geometry and Actions} and GIF Grant 1137-30.6/2011}

\newtheorem{theorem}{Theorem}

\newtheorem{lemma}{Lemma}

\newtheorem{proposition}{Proposition}

\newcommand{\C}{\mathbb C}
\newcommand{\M}{\mathcal M}
\newcommand{\Z}{\mathbb Z}
\newcommand{\op}{\operatorname}
\newcommand{\id}{\op{id}}

\theoremstyle{definition}
\newtheorem{definition}{Definition}
\newtheorem{remark}{Remark}

\numberwithin{theorem}{section}
\numberwithin{definition}{section}
\numberwithin{proposition}{section}
\numberwithin{lemma}{section}
\numberwithin{corollary}{section}
\numberwithin{equation}{section}
\numberwithin{remark}{section}

\begin{document}

\begin{abstract}
We determine the $K$-theory of the $C^{*}$-algebra $C(SU_{-1}(2))$ and describe its spectrum. Moreover, we exhibit a continuous $C^{*}$-bundle over $[-1,0)$ whose fibre at $q$ is isomorphic to $C(SU_{q}(2))$.
\end{abstract}

\maketitle


\setcounter{section}{-1}

\section{Introduction}
In the Woronowicz' theory of compact quantum groups \cite{Woronowicz1987a,Woronowicz1998}, $q$-defor\-ma\-tions of compact Lie groups serve as fundamental examples \cite{Woronowicz1987,Woronowicz1988}. In the algebraic setting, Drinfel'd and Jimbo introduced $q$-deformed semisimple Lie groups \cite{Drinfeld1986, Jimbo1985}. In the case of compact Lie groups, Rosso showed in \cite{Rosso1990} that the approaches of Woronowicz and Drinfel'd and Jimbo are essentially equivalent. Since the quantum group $SU_{q}(2)$ is a fundamental example of a $q$-deformation, it attracts a great deal of attention. It has been studied intensively from various perspectives, with most research focussing on the case of a positive deformation parameter $q$, see for example \cite{ChakrabortyPal2003,DabrowskiGiovanniSitarzSuilekomVarilly2005}. In particular, the $K$-theory for $C(SU_q(2))$ with positive parameter $q$ has been computed in \cite{MasudaNakagamiWatanabe1989}. Recently, the quantum groups $SU_{q}(2)$ for $q<0$ have attracted some attention as well, most notably because of the close relation to free orthogonal quantum groups $A_{o}(F)$ defined for $F\in GL(n,\mathbb C)$ with $F\bar{F}\in\mathbb R\cdot1$, \cite{Banica1997,BichonDeRijdtVaes2006,VanDaeleWang1996}. More specifically, we have
\[
SU_{-1}(2)\cong A_{o}(2)
\]
by \cite[Proposition 7]{Banica1997}. In the context of recent work on the Baum-Connes conjecture for free orthogonal groups \cite{Voigt2011}, and since $SU_{-1}(2)$ is the only free orthogonal group not treated in this setting, it is natural to ask for the $K$-theory of $C(SU_{-1}(2))$. In order to determine $K_*(C(SU_{-1}(2)))$, we take a different approach and use Zakrzewski's \cite{Zakrzewski1991} concrete realisation of $C(SU_{-1}(2))$ as a sub-$C^{*}$-algebra of $M_{2}(C(SU(2)))$.

Let us explain how this work is organized. We start with a preliminary section on compact quantum groups and on Zakrzewski's result, for which we present an alternative proof. In Section \ref{spectrum}, we determine the spectrum of $C(SU_{-1}(2))$. In Section \ref{K-theory}, we compute the $K$-theory of $C(SU_{-1}(2))$. As for $C(SU_{q}(2))$ with $q\neq-1$, it turns out that both $K$-groups of $C(SU_{-1}(2))$ are isomorphic to $\mathbb Z$ with generators being the class of the unit and the class of the canonical unitary in $M_{2}(C(SU_{-1}(2)))$, respectively. In the last section, we use the Haar state on $C(SU_{-1}(2))$ to prove the existence of a continuous $C^*$-bundle over $[-1,0)$ whose fibre at $q$ is isomorphic to $C(SU_{q}(2))$. This is reminiscent of Blanchard's result \cite{Blanchard1996} that the positive $q$-deformations of $SU(2)$ form a continuous $C^*$-bundle over $(0,1]$.

This article is based on the author's diploma thesis.  We remark that at least parts of the above results are certainly known to the experts. However, to the best of our knowledge they are not documented in the literature.

I would like to thank the advisor of my diploma thesis, Christian Voigt, for fruitful discussions and his good advice. Moreover, I would like to thank Dominic Enders, Sven Raum, and Christoph Winges for useful comments on earlier versions of this article. I thank Adam Skalski who informed me of Zakrzewski's article.

\section{Preliminaries}
\label{prelim}

We begin by recalling the definition of a compact quantum group, which goes back to Woronowicz. For a detailed treatment of quantum groups, we refer to the literature \cite{KlimykSchmuedgen1997,Timmermann2008}. When dealing with tensor products, we always use the minimal tensor product of $C^{*}$-algebras.

\begin{definition}
Let $A$ be a unital $C^{*}$-algebra and $\Delta:A\to A\otimes A$ a unital $*$-homomorphism. The pair $(A,\Delta)$ is called a \emph{compact quantum group} if the following conditions hold:
\begin{enumerate}
	\item[i)] $\Delta$ is \emph{coassociative}, i.e.\ $(\Delta \otimes \id) \circ  \Delta = (\id \otimes\Delta ) \circ \Delta$,
	\item[ii)] $(A,\Delta)$ is \emph{bisimplifiable}, i.e.\ $\Delta(A)(A\otimes 1)$ and $\Delta(A) (1\otimes A)$ are linearly dense in $A\otimes A$.
\end{enumerate}
The $*$-homomorphism $\Delta$ is called the \emph{comultiplication} of the compact quantum group.
\end{definition}

If $G$ is a compact group, $C(G)$ is a compact quantum group with comultiplication
\[
\Delta:C(G)\longrightarrow C(G)\otimes C(G) \cong C(G\times G), \quad \Delta(f)(s,t) := f(s\cdot t).
\]
In the special case where $G$ is a closed subgroup of the unitary group $\mathcal U(n)$, the comultiplication satisfies
\[
\Delta(u_{ij})=\sum\limits_{k=1}^{n} u_{ik} \otimes u_{kj},
\]
with $u_{ij}$ denoting the canonical projection onto the $(i,j)$-th coordinate.

\begin{definition}[\cite{Woronowicz1987}]
For $q\in[-1,1]$, let $C(SU_{q}(2))$ be the universal unital $C^{*}$-algebra with generators $\alpha$ and $\gamma$ satisfying the relations
\begin{equation}
\label{relationsSU-1}
\begin{split}
\alpha^{*}\alpha+\gamma^{*}\gamma= 1,\quad \alpha\alpha^{*} +  q^{2} \gamma\gamma^{*} &= 1 ,\\
\alpha\gamma = q \gamma\alpha , \quad \alpha\gamma^{*} = q\gamma^{*}\alpha ,\quad \gamma\gamma^{*} &= \gamma^{*}\gamma.
\end{split}
\end{equation}
\end{definition}

Note that the above relations imply, that
\[
u_{q}:=\begin{pmatrix}\alpha&-q\gamma^{*}\\\gamma & \alpha^{*}\end{pmatrix} \in M_{2}(C(SU_{q}(2)))
\]
is a unitary. The comultiplication $\Delta_{q}:C(SU_{q}(2))\rightarrow C(SU_{q}(2))\otimes C(SU_{q}(2))$ given by
\[
\Delta_{q}(\alpha) := \alpha\otimes \alpha  -  q\gamma^{*}\otimes \gamma\quad \text{and}\quad \Delta_{q}(\gamma) := \gamma\otimes \alpha  +  \alpha^*\otimes\gamma
\]
turns $(C(SU_{q}(2)),\Delta_{q})$ into a compact quantum group for $q\neq0$. It has been shown in \cite{Woronowicz1987} that for $q\neq 0$,
\[
\mathcal B_q := \left\lbrace \alpha^{k}\gamma^{l}{\gamma^{*}}^{m},\ {\alpha^{*}}^{p}\gamma^{l}{\gamma^{*}}^{m} \ : \ k,l,m\in\mathbb N_{0},\ p\in \mathbb N \right\rbrace
\]
is a linearly independent set spanning the canonical dense $*$-subalgebra of $C(SU_q(2))$. Of course, we agree on the convention that $\alpha^0:=1$ etc.. Define
\[
\eta_{klm}:=\begin{cases}\alpha^{k}\gamma^{l}{\gamma^{*}}^{m}& ,\quad \text{if}\ k\geq0,\ m,n\geq0,\\ {\alpha^*}^{-k}\gamma^{l}{\gamma^{*}}^{m}& ,\quad \text{if}\ k<0,\ m,n\geq0.\end{cases}
\]

Observe that for $q=\pm 1$ the relations \eqref{relationsSU-1} are symmetric in $\alpha$ and $\gamma$. Furthermore, $C(SU_1(2))$ is a commutative $C^*$-algebra with spectrum homeomorphic to 
\[
S^3=\left\lbrace(a,c)\in \C^{2}\ :\ |a|^{2}+|c|^{2}=1 \right\rbrace,
\]
and the homeomorphism 
\[
S^3 \stackrel{\cong}{\longrightarrow} SU(2),\quad (a,c)\mapsto \begin{pmatrix}a&-\bar{c}\\c& \bar{a}\end{pmatrix}
\]
induces an isomorphism
\[
\psi:C(SU_1(2))\longrightarrow C(SU(2)),\quad \psi(\alpha)((a,c)):=a,\ \psi(\gamma)((a,c)):=c.
\]
Moreover, $\psi$ fits into a commutative diagram
\[
	\xymatrix{
	C(SU_1(2)) \ar[r]^/-0.6cm/{\Delta_1} \ar[d]^\psi & 	C(SU_1(2))\otimes C(SU_1(2)) \ar[d]^{\psi\otimes \psi}\\
	C(SU(2)) \ar[r]^/-0.6cm/{\Delta} & C(SU(2))\otimes C(SU(2))	
	}
\]
and thus induces an isomorphism $(C(SU_1(2)),\Delta_1)\cong (C(SU(2)),\Delta)$ of compact quantum groups. We shall use this identification throughout this work.

\begin{definition}
Let $(A,\Delta)$ be a compact quantum group. A state $h$ on $A$ is called \emph{left invariant} if
\[
(\id\otimes h)\circ \Delta(a)=h(a)\cdot 1\quad \text{for all}\ a\in A.
\]
It is called \emph{right invariant} if
\[
(h\otimes \id) \circ \Delta(a) = h(a)\cdot 1\quad \text{for all}\ a\in A.
\]
\end{definition}

Every compact quantum group has a unique left and right  invariant state \cite{VanDaele1995}, called the \emph{Haar state}. For $q\in (-1,1]\setminus\left\lbrace0\right\rbrace$, the Haar state $h_q$ of $C(SU_q(2))$ can be characterized as follows.

\begin{proposition}
\label{HaarFaithful}
For $q\in (-1,1)\setminus\left\lbrace0\right\rbrace$, the Haar state $h_{q}$ is the unique state on $C(SU_{q}(2))$ satisfying
\[
h_{q}(\eta_{klm})=
\begin{cases}
\frac{1 - q^2}{1 - q^{2m+2}}& , \quad \text{if}\ k=0,\ l=m,\\
0 & , \quad \text{otherwise.}
\end{cases}
\]
For $q=1$, the Haar state $h_{1}$ is uniquely determined by 
\[
h_{1}(\eta_{klm})=\begin{cases}
\frac{1}{m+1} & ,\quad \text{if}\ k=0,\ l=m ,\\			0 & , \quad \text{otherwise.}	
\end{cases}
\]
Furthermore, all the Haar states are faithful.
\end{proposition}
\begin{proof}
For $|q|<1$ see \cite[4.3, Theorem 14]{KlimykSchmuedgen1997} for the characterisation and \cite{Nagy1993} for the faithfulness of $h_{q}$. For $q=1$, the Haar integral on $C(SU(2))$ coincides with the Haar state. By applying to the spherical coordinate system, one can then verify the above equalities directly.
\end{proof} 

For $(a,c)\in SU(2)$, we define a representation $\pi_{(a,c)}$ of $C(SU_{-1}(2))$ as follows. If $c=0$, let 
\[
\pi_{(a,0)}: C(SU_{-1}(2))\longrightarrow \mathbb C, \quad \pi_{(a,0)}(\alpha):=a, \ \pi_{(a,0)}(\gamma):=0,
\]
and similarly, set
\[
\pi_{(0,c)}: C(SU_{-1}(2))\longrightarrow \mathbb C,\quad \pi_{(0,c)}(\alpha):=0,\ \pi_{(0,c)}(\gamma):=c
\]
if $a=0$. In all other cases, define
\[
\begin{array}{l}
\pi_{(a,c)}: C(SU_{-1}(2)) \longrightarrow M_2(\mathbb C),\\
\pi_{(a,c)}(\alpha):=\begin{pmatrix}a&0\\0&-a \end{pmatrix}, \ \pi_{(a,c)}(\gamma):=\begin{pmatrix}0&c\\c&0 \end{pmatrix}.
\end{array}
\]
Each of these representations is surjective and hence irreducible. Also note that every character of $C(SU_{-1}(2))$ is of the form $\pi_{(a,c)}$ with $a$ or $c$ being zero.

\begin{proposition}
\label{reps}
Every irreducible representation of $C(SU_{-1}(2))$ is unitarily equivalent to some $\pi_{(a,c)}$.
\end{proposition}
\begin{proof}
Let $\pi:C(SU_{-1}(2))\to \mathcal B(H)$ be an irreducible representation on a Hilbert space $H$ not isomorphic to $\mathbb C$. The relations \eqref{relationsSU-1} imply that $\alpha^{2}$ and $\gamma^{2}$ are central. Hence $\pi(\alpha^{2})$ and $\pi(\gamma^{2})$ are scalar, whereas irreducibility prevents $\pi(\alpha)$ and $\pi(\gamma)$ from being scalar. Thus the spectrum of $\pi(\alpha)$ consists of two distinct points, some $a\in \mathbb C$ with $|a|\leq1$ and its negative. The same holds for the spectrum of $\pi(\gamma)$ and some $c\in \mathbb C$ with $|c|\leq1$. By functional calculus, there are projections $P$, $Q$ in the image of $\pi$ such that
\[
\pi(\alpha) = aP - a(1-P) = 2aP - a1  \quad \text{and}\quad \pi(\gamma) = 2cQ - c1 .
\]
In fact, the image of $\pi$ coincides with $C^{*}(P,Q,1)$, which itself is the image of an irreducible representation of the unital universal $C^{*}$-algebra generated by two projections. This is a 2-subhomogeneous $C^{*}$-algebra \cite[IV.1.4.2]{Blackadar2005}, and $\pi$ turns out to be a surjective $*$-homomorphism $\pi:C(SU_{-1}(2))\rightarrow M_{2}(\mathbb C)$. Since $\alpha$ is normal, we may assume $\pi(\alpha)$ to be of the form
\[
\pi(\alpha)=\begin{pmatrix}a&0\\0&-a\end{pmatrix} .
\]
The fact that $\alpha$ and $\gamma$ anticommute implies that $\pi(\gamma)$ is an off-diagonal matrix. Moreover, as $\pi(\gamma\gamma^{*})$ is scalar, we also get that the off-diagonal entries have the same modulus. Conjugating $\pi$ with a suitable diagonal unitary, we may also assume that
\[
\pi(\gamma)=\begin{pmatrix}0&c\\c&0\end{pmatrix}.
\]
Observe that $\pi(\alpha)$ is left invariant under conjugation with a diagonal unitary. Finally, $\alpha\alpha^{*}+\gamma\gamma^{*}=1$ implies that $(a,c)\in SU(2)$, and thus $\pi=\pi_{(a,c)}$.
\end{proof}

In \cite{Zakrzewski1991}, Zakrzewski showed that $C(SU_{-1}(2))$ can be represented as a sub-$C^{*}$-algebra of $M_{2}(C(SU(2)))$. More generally, he shows that a matrix $2\times2$ quantum group can be considered as a sub-$C^{*}$-algebra of the $C^{*}$-algebra of continuous functions on the classical group with values in $M_{2}(\mathbb C)$. For the convenience of the reader, we shall give an alternative proof for the concrete realisation of $C(SU_{-1}(2))$ using the irreducible representations $\pi_{(a,c)}$.

By the universal property of $C(SU_{-1}(2))$, there is a $*$-homomorphism 
\[
\begin{array}{l}
\phi:C(SU_{-1}(2))\longrightarrow M_{2}(C(SU(2))),\\
\phi(\alpha)((a,c)):=\begin{pmatrix}a & 0\\ 0 & -a\end{pmatrix}, \ \phi(\gamma)((a,c)):=\begin{pmatrix}0 & c \\ c &  0\end{pmatrix}.
\end{array}
\]
Define $G:=\Z/2\Z \oplus  \Z/2\Z$, $r:=(1,0)$, and $s:=(0,1)$, and consider the $G$-action on $C(SU(2))$ given by 
\[
(f\cdot r)(a,c):=f(-a,c)\quad \text{and}\quad (f\cdot s)(a,c) := f(a,-c).
\]
This induces an action $\beta:G\curvearrowright M_{2}(C(SU(2)))$ given by
\[
\beta_r\left(\begin{pmatrix}f&g\\h&k\end{pmatrix}\right):=\begin{pmatrix}k\cdot r&h\cdot r\\g\cdot r&f\cdot r\end{pmatrix} \ \text{and}\
\beta_s\left(\begin{pmatrix}f&g\\h&k\end{pmatrix}\right):=\begin{pmatrix}f\cdot s&-g\cdot s\\-h\cdot s&k\cdot s\end{pmatrix}.
\]

\begin{theorem}
\label{fixed point}
The $*$-homomorphism $\phi$ is injective and maps onto the fixed point algebra of $\beta$.
\end{theorem}
\begin{proof}
Let $\op{ev}_{(a,c)}\in C(SU(2))$ be the evaluation at $(a,c)\in SU(2)$. For non-zero $a$ and $c$, 
\[
\op{ev}_{(a,c)}\circ\phi=\pi_{(a,c)}. 
\]
On the other hand, for $a,c\in \mathbb T$,
\[
\op{ev}_{(a,0)}\circ\phi=\pi_{(a,0)}\oplus\pi_{(-a,0)}\quad \text{and}\quad \op{ev}_{(0,c)}\circ \phi=\op{Ad}(v)\circ \pi_{(0,c)}\oplus \pi_{(0,-c)},
\]
where $v\in M_2(\mathbb C)$ is the symmetry given by 
\[
v:=\frac{1}{\sqrt{2}}\cdot \begin{pmatrix}1&1\\1&-1 \end{pmatrix}.
\]
For $0\neq x\in C(SU_{-1}(2))$, we therefore find $(a,c)\in SU(2)$ such that $\op{ev}_{(a,c)}\circ\phi(x)\neq0$. This shows that $\phi$ is injective. 

It is obvious that the image of $\phi$ is contained in the fixed point algebra of $\beta$. For the other inclusion, note that every element in $x\in M_{2}(SU(2))$ admits a unique decomposition
\[
x=\begin{pmatrix}f&0\\0&f \end{pmatrix}+\begin{pmatrix}g&0\\0&-g \end{pmatrix}+\begin{pmatrix}0&h\\h&0 \end{pmatrix}+\begin{pmatrix}0&k\\-k&0 \end{pmatrix}
\]
for some $f,g,h,k \in C(SU(2))$. One easily verifies that $x$ is fixed by $\beta_{1}$ and $\beta_{2}$ if and only if this holds for each summand. In this case,
\[
\begin{array}{lcl}
f = f\cdot r = f\cdot s, & & g = -g\cdot r = g\cdot s,\\
h = h\cdot r = -h\cdot s, & & k = -k\cdot r = -k\cdot s.
\end{array}
\]
Note that these relations are induced by (different) $G$-actions on $C(SU(2))$. Hence, the coefficient functions can be approximated by linear combinations of elements in $\mathcal B_1$ satisfying the same symmetry relations. It follows that $C(SU_{-1}(2))$ is mapped onto the fixed point algebra of $\beta$.
\end{proof}

\section{\texorpdfstring{The spectrum of $C(SU_{-1}(2))$}{Spectrum}}
\label{spectrum}

In this section, we determine the spectrum of $C(SU_{-1}(2))$. As a 2-sub\-homogeneous $C^{*}$-algebra is GCR, the canonical surjection
\[
C(SU_{-1}(2))\widehat{\quad}\longrightarrow  \operatorname{Prim}(C(SU_{-1}(2)))
\]
is injective \cite[IV.1.3.6 and IV.1.3.7]{Blackadar2005}. Hence, with this identification, the topology on the spectrum on $C(SU_{-1}(2))$ coincides with the  Jacobson topology on $\operatorname{Prim}(C(SU_{-1}(2)))$. Denote by
\[
X := SU(2)/\sim
\]
the quotient space by the following equivalence relation. If $(a,c)\in SU(2)$ satisfies $ac\neq0$, then
\[
(a,c) \sim (b,d)\Leftrightarrow (b,d)\in G\cdot(a,c).
\]
If $a$ or $c$ is zero, the equivalence class of $(a,c)$ only consists of one element. We write
\[
P:SU(2)\longrightarrow X
\]
for the canonical projection and endow $X$ with the final topology with respect to $P$. The map
\[
\Pi : X \longrightarrow C(SU_{-1}(2))\widehat{\quad}, \quad [(a,c)] \mapsto [\pi_{(a,c)}] 
\]
is a bijection. To see this, consider the symmetries $v_0,v_1\in M_{2}(\mathbb C)$ given by
\begin{align}
\label{symmetries}
v_0:=\begin{pmatrix}1&0\\0&-1 \end{pmatrix}\quad \text{and}\quad v_1 := \begin{pmatrix}0&1\\1&0 \end{pmatrix}.
\end{align}
Every two-dimensional representation $\pi_{(a,c)}$ gets intertwined with $\pi_{(-a,c)}$ by $v_1$, and similarly $v_0$ intertwines $\pi_{(a,c)}$ and $\pi_{(a,-c)}$. Thus, we also get the equivalence of $\pi_{(a,c)}$ and $\pi_{(-a,-c)}$. This shows that $\Pi$ is well-defined and surjective. One can check that for $[(a,c)]\neq [(b,d)]\in X$ either
\[
\op{det}(\pi_{(a,c)}(\alpha))\neq \op{det}(\pi_{(b,d)}(\alpha)),
\]
or
\[
\op{det}(\pi_{(a,c)}(\gamma))\neq \op{det}(\pi_{(b,d)}(\gamma)),
\]
where $\op{det}$ denotes the respective determinant. This yields that $\Pi$ is injective.

\begin{theorem} 
The bijection $\Pi$ is a homeomorphism of $X$ onto the spectrum of $C(SU_{-1}(2))$.
\end{theorem}
\begin{proof}
Recall the relationship between $\op{ev_{(a,c)}}$ and $\pi_{(a,c)}$ under the identification via $\phi$ described in the proof of Theorem \ref{fixed point}.
Let $M\subseteq X$ and define $N:=\Pi(M)\subseteq C(SU_{-1}(2))^{\widehat{\quad}}$. We show that $\Pi(\overline{M})=\overline{N}$. Let $x\in C(SU_{-1}(2))\subseteq C(SU(2),M_{2}(\mathbb C))$ be an element satisfying $\pi_{(a,c)}(x)=0$ for all $(a,c)\in P^{-1}(M)\subseteq SU(2)$. Since the coefficient functions of $x$ are continuous, $\pi_{(a,c)}(x)=0$ for all $(a,c)\in \overline{P^{-1}(M)}=P^{-1}(\overline{M})$. Hence, $\Pi(\overline{M})\subseteq \overline{N}$.

For the other implication, assume that $\emptyset\neq \overline{M}\neq X$ and take an arbitrary element $[(b,d)]\in X\setminus\overline{M}$. We show that $\Pi([(b,d)])\notin \overline{N}$. Suppose first that $b$ and $d$ are non-zero. Denote by $Y\subseteq SU(2)$ the $G$-Orbit of $P^{-1}(\overline{M})\cup \left\lbrace(b,d)\right\rbrace$. Let $y\in C(Y,M_{2}(\mathbb C))$ be given by
\[
y_{|G\cdot P^{-1}(\overline{M})} = 0\quad \text{and}\quad y_{|G\cdot (b,d)} = \begin{pmatrix}1&0\\0&1 \end{pmatrix}.
\]
Obviously, $y$ lies in the fixed point algebra of the $G$-action on $C(Y,M_{2}(\mathbb C))$ induced by $\beta$. Find a $\beta$-invariant lift $x\in C(SU_{-1}(2))\subseteq C(SU(2),M_{2}(\mathbb C))$ for $y$. By construction, $x\in \operatorname{ker}(\pi_{(a,c)})$ for every $(a,c)\in P^{-1}(\overline{M})$, but $x\notin \operatorname{ker}(\pi_{(b,d)})$. If $d=0$ and $(-b,0)\notin P^{-1}(\overline{M})$, then we can proceed as before. Let us therefore assume that $d=0$ and $(-b,0)\in P^{-1}(\overline{M})$. Define the closed subset
\[
Y := \left\lbrace (a,\pm c)\in SU(2)\ : \ (a,c)\in P^{-1}(\overline{M}) \right\rbrace,
\]
and let $f\in C(SU(2))$ be a function with $f_{|Y}=0$ and $f(b,0)=1$. Now,
\[
x := \begin{pmatrix}\frac{1}{2} (f+f\cdot s)&0\\0&\frac{1}{2} (f+f\cdot s)\cdot r\end{pmatrix} \in C(SU_{-1}(2))
\]
satisfies $\pi_{(b,0)}(x)=1$ and $\pi_{(a,c)}(x)=0$ for all $(a,c)\in P^{-1}(\overline{M})\subseteq Y$. If $b=0$, we can use a similar argument. We have shown that $\Pi([(b,d)])\notin \overline{N}$, and the proof is complete.
\end{proof}

\section{\texorpdfstring{$K$-Theory of $C(SU_{-1}(2))$}{K-theory}}
\label{K-theory}

Consider the following closed subsets of $SU(2)$:
\[
\begin{array}{lcl}
X_{1} & := & \left\lbrace (1,0), (0,1)\right\rbrace, \\
X_{2} & := & \left\lbrace (a,c)\in SU(2) \ : \ \op{Im}(a)  =  \op{Im}(c) =  0  \right\rbrace, \\
X_{3} & := & \left\lbrace (a,c)\in SU(2) \ :\   \op{Im}(a),\, \op{Im}(c) \geq 0 \ \text{and}\ \op{Im}(a)\cdot \op{Im}(c) = 0 \right\rbrace, \\
X_{4} & := & \left\lbrace (a,c)\in SU(2) \ : \ \op{Im}(a),\, \op{Im}(c) \geq 0  \right\rbrace.
\end{array}
\]
These subsets obviously define a filtration of $SU(2)$
\[
X_1\subseteq X_2\subseteq X_3\subseteq X_3\subseteq SU(2),
\]
and the induced restriction homomorphisms give rise to a cofiltration of $C^*$-algebras
\[
	\xymatrix{
	 C(SU_{-1}(2)) \ar@{->>}[r]^/1em/{\pi_4} & A_4 \ar@{->>}[r]^{\pi_3} &  A_3 \ar@{->>}[r]^{\pi_2} & A_2 \ar@{->>}[r]^{\pi_1} &  A_{1}.
	}
\]
In order to compute the $K$-theory of $C(SU_{-1}(2))$, we successively compute $K_*(A_{k+1})$ using $K_*(A_k)$ and the six-term exact sequence associated with $\pi_k$. Observe that every coefficient function of an element in the fixed point algebra of $\beta$ is uniquely determined by its values on $X_{4}$. Hence, $\pi_{4}$ is an isomorphism, and we only have to determine $K_*(A_4)$. For any subset $M\subset SU(2)$, we extend the notation of the last section and write $f\cdot r$ and $f\cdot s$ for $f\in C(M)$, whenever it makes sense.

Recall the symmetries $v_0$, $v_1\in M_2(\mathbb C)$ defined in Section \ref{spectrum}. We have that
\[
A_1 \cong C^{*}(v_0) \oplus C^{*}(v_1) \cong  \mathbb C^{2} \oplus \mathbb C^{2},
\]
and
\[
A_{2} = \left\lbrace \begin{pmatrix}f&g\\g\cdot r&f\cdot r\end{pmatrix}\in M_{2}(C(X_{2})) \ : \ f=f\cdot s,\, g=-g\cdot s \right\rbrace.
\]
An element in $A_2$ is already uniquely determined by its values on the compact subset of $X_{2}$ consisting of all $(a,c)$ where $\op{Re}(a)$ and $\operatorname{Re}(c)$ are non-negative. Using a homeomorphism between this subset and $[0,1]$ sending $(1,0)$ to $0$ and $(0,1)$ to $1$, we get an isomorphism between $A_2$ and
\[
A_2\stackrel{\cong}{\longrightarrow } \left\lbrace f\in C([0,1],M_{2}(\mathbb C)) \ : \ f(0)\in C^{*}(v_0),\, f(1) \in C^{*}(v_1) \right\rbrace =:C.
\]

\begin{lemma} 	
\label{A2}
We have that $K_{0}(C)\cong\mathbb Z^{3}$ and $K_{1}(C)=0$. Moreover, $K_{0}(C)$ is generated by $[p_{0}], [q_{0}]$, and $[1]$, where for $t\in [0,1]$, the projections $p_{0}(t), q_{0}(t)\in M_{2}(\mathbb C)$ are given by
\[
p_{0}(t) := \begin{pmatrix}1 - \frac{t}{2}&\sqrt{\frac{t}{2} - \frac{t^2}{4}}\\\sqrt{\frac{t}{2} - \frac{t^2}{4}}&\frac{t}{2} \end{pmatrix} \quad \text{and} \quad  q_{0}(t) := \begin{pmatrix}\frac{t}{2}&\sqrt{\frac{t}{2} - \frac{t^2}{4}}\\\sqrt{\frac{t}{2} -  \frac{t^2}{4}}&1 - \frac{t}{2} \end{pmatrix}.
\]
\end{lemma}
\begin{proof}
The kernel of $\op{ev}_{0} \oplus  \op{ev}_{1}: C\to C^{*}(v_0) \oplus C^{*}(v_1)$ is $C_{0}((0,1),M_{2}(\mathbb C))$, and we get the following six-term exact sequence
\[
	\xymatrix{
	0 \ar[r] & K_{0}(C) \ar[r] & K_{0}(C^{*}(v_0)) \oplus  K_{0}(C^{*}(v_1)) \cong \mathbb Z^{4} \ar[d]^{\rho} \\					
	0 \ar[u] & K_{1}(C) \ar[l] & K_{1}(C_{0}((0,1))) \cong \mathbb Z \ar[l]
			}
\]
As generators for $K_{0}(C^*(v_0))\oplus K_{0}(C^*(v_1))$, we may choose
\[
\begin{array}{lcl}
e_{1} := \left(\left[\begin{pmatrix}1&0\\0&0\end{pmatrix}\right], 0\right), & & e_{3} := \left( 0, \left[\frac{1}{2}\cdot\begin{pmatrix}1&1\\1&1\end{pmatrix}\right] \right),\\
e_{2} := \left( \left[\begin{pmatrix}0&0\\0&1\end{pmatrix} \right], 0\right), & & e_{4} := \left( 0, \left[\frac{1}{2}\cdot\begin{pmatrix}1&-1\\-1&1\end{pmatrix}\right]\right).
\end{array}
\]
To compute the images of these elements under the exponential map $\rho$, set
\[
\begin{array}{lcl}
f_{1}(t) := (1-t)\cdot\begin{pmatrix}1&0\\0&0\end{pmatrix}, & &   
f_{3}(t) := \frac{t}{2}\cdot\begin{pmatrix}1&1\\1&1\end{pmatrix}, \\
f_{2}(t) := (1-t)\cdot\begin{pmatrix}0&0\\0&1\end{pmatrix}, & & 
f_{4}(t) := \frac{1}{2}\cdot\begin{pmatrix}1&-1\\-1&1\end{pmatrix}
\end{array}
\]
for $t\in [0,1]$. For $1\leq j\leq4$, $\rho(e_{j})=\left[ \, \exp(2\pi i f_{j}) \, \right]$, and
\[
\rho(e_{1}) = \rho(e_{2}) =  -\rho(e_{3}) =  -\rho(e_{4}) =  -[z],
\]
where $z\in C(\mathbb T)$ denotes the identity map on $\mathbb T$. Thus, $\rho$ is surjective, and
\[
\op{ker}(\rho) = \langle e_{1}+e_{3},\,  e_{1}-e_{2},\,  e_{3}-e_{4} \rangle.
\]
Moreover,
\[
\begin{array}{lcl}
K_{0}(\op{ev}_0\oplus \op{ev}_1)([p_{0}]) = e_{1}+e_{3} , & & \hspace{-1.25cm} K_{0}(\op{ev}_0\oplus \op{ev}_1)([p_0]-[q_0]) = e_{1}-e_{2}, \\
K_0(\op{ev}_0\oplus \op{ev}_1)([p_0]+[q_0]-[1]) = e_3-e_4.
\end{array}
\]
The claim now follows from exactness of the above six-term exact sequence.
\end{proof}

Let us denote by $\tilde{p}$, $\tilde{q}\in A_{2}$ the unique extensions of $p_{0}$, $q_{0}\in C$. Note that these projections satisfy the relations
\[
\begin{array}{lcl}
\tilde{p}(1,0) = \tilde{q}(-1,0)=\begin{pmatrix}1&0\\0&0\end{pmatrix}, & & \tilde{p}(0,1) = \tilde{q}(0,1) = \frac{1}{2}\cdot\begin{pmatrix} 1 & 1 \\ 1 & 1 \end{pmatrix},\\
\tilde{p}(-1,0) = \tilde{q}(1,0) = \begin{pmatrix}0&0\\0&1\end{pmatrix}, & & \tilde{p}(0,-1) = \tilde{q}(0,-1) =\frac{1}{2}\begin{pmatrix}1&-1\\-1&1\end{pmatrix}.
\end{array}
\]

In order to compute the $K$-theory of $A_{3}$, write $X_{3}$ as the disjoint union
\[
X_{3} = X_{2} \cup U_{1} \cup U_{2},
\]
with
\[
\begin{array}{l}
U_{1} := \left\lbrace (a,c)\in SU(2) \ : \ \op{Im}(a)=0,\, \op{Im}(c)>0 \right\rbrace,\\
U_{2} := \left\lbrace (a,c)\in SU(2) \ : \ \op{Im}(a)>0,\, \op{Im}(c)=0 \right\rbrace.
\end{array}
\]
With this decomposition, it is obvious that $\op{ker}(\pi_{2})\subseteq M_{2}(C_{0}(U_{1}\cup U_{2}))$. More precisely, we get an exact sequence 
\begin{align}
\label{exactA3}
\begin{xy}
\xymatrix{
	0 \ar[r] & I_{1}\oplus I_{2} \ar[r]& A_{3} \ar[r]^{\pi_{2}} & A_{2} \ar[r] & 0,
		}
\end{xy}
\end{align}
where
\[
\begin{array}{l}
I_{1} := \left\lbrace \begin{pmatrix}f&g\\h&k\end{pmatrix} \in M_{2}(C_{0}(U_{1})) \ : \ f=k\cdot s,\, g=h\cdot s \right\rbrace \\
I_{2} := \left\lbrace \begin{pmatrix}f&g\\h&k\end{pmatrix}\in M_{2}(C_{0}(U_{2})) \ : \ \begin{matrix}f=f\cdot r,\, g=-g\cdot r, \\ h=-h\cdot r,\,  k=k\cdot r \end{matrix} \right\rbrace.
\end{array}
\]
Note that $I_{1}$ and $I_{2}$ are isomorphic via the automorphism of $C(SU_{-1}(2))$ which exchanges $\alpha$ and $\gamma$. If $\mathbb D\subseteq \mathbb C$ denotes the closed unit disk, then there is a homeomorphism 
\[
U_{2} \longrightarrow \mathbb D\setminus \mathbb T, \quad (a+ib,c)\mapsto a+ic.
\]
An element in $I_{2}$ is uniquely determined by its values on elements $(a,c)\in U_{2}$ with $c\geq0$. Hence, the above homeomorphism induces an isomorphism
\[
I_2\stackrel{\cong}{\longrightarrow} \left\lbrace f\in C_{0}(\mathbb D\setminus \mathbb T_+,\, M_{2}(\mathbb C)) \ : \ f(x)\in C^*(v_0),\, x\in \mathbb T\setminus\mathbb T_+ \right\rbrace,
\]
where $\mathbb T_+\subseteq \mathbb T$ consists of all elements with modulus one and non-negative real part. As a consequence, $I_{2}$ is homotopy equivalent to
\[
I := \left\lbrace  f\in C_{0}(\mathbb D\setminus \left\lbrace1\right\rbrace,\,  M_{2}(\mathbb C))
\ : \ f(x)\in C^{*}(v_0),\, x\in \mathbb T\setminus\left\lbrace 1\right\rbrace \right\rbrace.
\]
The $C^{*}$-algebras $I_1$, $I_2$, and $I$ all have isomorphic $K$-theory, which we now determine by using the exact sequence
\begin{align}
\label{extensionI}
\begin{xy}
	\xymatrix{
	 0 \ar[r]& C_{0}(\mathbb D\setminus \mathbb T,\, M_{2}(\mathbb C)) \ar[r] &  I \ar[r] & C_{0}((0,1),\, C^{*}(v_0)) \ar[r] & 0
	}
\end{xy}
\end{align}
induced by restriction on $\mathbb T\setminus\left\lbrace 1\right\rbrace$.

\begin{lemma} 	
\label{I}
We have $K_{0}(I)=0$ and $K_{1}(I)\cong\mathbb Z$. Moreover, if
\[
\rho: K_{1}(C_{0}((0,1),\, C^*(v_0))) \longrightarrow K_0(C_0(\mathbb D\setminus \mathbb T))
\]
denotes the index map associated with \eqref{extensionI}, then 
\[
K_1(I)\longrightarrow 
\op{ker}(\rho) = \mathbb Z\cdot \left[\op{diag}(z,\bar{z})\right]
\]
is an isomorphism.
\end{lemma}
\begin{proof}
Consider the exact six-term sequence associated with \eqref{extensionI}
\[
	\xymatrix{
	 \mathbb Z \cong K_{0}(C_{0}(\mathbb D\setminus \mathbb T)) \ar[r] & K_{0}(I) \ar[r] & 0 \ar[d] \\
	 \mathbb Z^{2} \cong  K_{1}(C_0((0,1),\, C^*(v_0))) \ar[u]^\rho & K_{1}(I) \ar[l] & 0 \ar[l]
	}
\]
As generators for $K_{1}(C_{0}((0,1),\, C^{*}(v_0)))$, we choose
\[
e_{1} := \left[\begin{pmatrix}z&0\\0&1\end{pmatrix}\right] \quad \text{and} \quad e_{2} := \left[\begin{pmatrix}1&0\\0&z\end{pmatrix}\right].
\]
It suffices to show that $\rho$ sends $e_{1}$ and $e_{2}$ to the Bott element. We only show this for $e_{1}$, the other case is similar. Let $I^{\sim}$ be the unitization of $I$ and $u \in M_{2}(I^\sim)$ the unitary lift for $\operatorname{diag}(z,1,\bar{z},1)$ given by 
\[
u(t) := \begin{pmatrix} t&0&-\sqrt{1-|t|^{2}}&0\\
						0&1&0&0\\
						\sqrt{1-|t|^{2}}&0&\bar{t}&0\\
						0&0&0&1
\end{pmatrix}, \quad t \in \mathbb D.
\]
If $1_{2}\in M_{2}(I^\sim)$ denotes the unit, then
\[
\begin{array}{lcl}
\rho(e_{1})  & = & [u\cdot \op{diag}(1_{2},0)\cdot u^{*}] - [1_{2}] \\
 & = & \left[\begin{pmatrix} |z|^{2}&z\sqrt{1-|z|^{2}}\\ \bar{z}\sqrt{1-|z|^{2}}&1-|z|^{2} \end{pmatrix}\right] - \left[\begin{pmatrix} 1&0\\0&0\end{pmatrix}\right],
\end{array}
\]
which is the Bott element in $K_0(C_0(\mathbb D\setminus \mathbb T))$. This completes the proof.
\end{proof}

\begin{remark}
\label{I1I2}
Consider $V_1$, $V_2\subseteq SU(2)$ defined as
\[
V_{1} := \left\lbrace (a,c)\in U_{1} \ : \ a=0 \right\rbrace \quad \text{and} \quad V_{2} := \left\lbrace  (a,c)\in U_{2} \ : \ c=0 \right\rbrace.
\]
Lemma \ref{I} shows that the homomorphisms
\[
K_{1}(I_{1}) \longrightarrow K_{1}(C_0(V_1,\, C^*(v_1))) \quad \text{and}\quad  K_1(I_2) \longrightarrow K_1(C_0(V_2,\, C^*(v_0))),
\]
induced by the inclusions $V_1\subseteq U_1$ and $V_2\subseteq U_2$, are injective.
\end{remark}

We proceed by computing the $K$-theory for $A_{3}$.

\begin{lemma}
\label{A3}
It holds that $K_{0}(A_{3})=\mathbb Z\cdot [1]$ and $K_{1}(A_{3})\cong\mathbb Z/2\Z$.
\end{lemma}
\begin{proof}
The exact sequence \eqref{exactA3} and Lemma \ref{I} induce the following six-term exact sequence
\begin{align}
\label{sixTermA3}
\begin{xy}
	\xymatrix{
	0 \ar[r] & K_{0}(A_{3}) \ar[r]^/-0.3cm/{K_{0}(\pi_{3})} & K_{0}(A_{2}) \cong \mathbb Z^{3} \ar[d]^\rho\\
	0 \ar[u] & K_{1}(A_{3}) \ar[l] & K_{1}(I_{1})\oplus K_{1}(I_{2}) \cong  \mathbb Z\oplus \mathbb Z\ar[l]
			}
\end{xy}
\end{align}
It is clear from Lemma \ref{A2} that $[1]\in K_{0}(A_{3})$ generates a copy of $\mathbb Z$. So it suffices to determine the images of $[\tilde{p}],[\tilde{q}]\in K_0(A_2)$ under $\rho$. Set
\[
t(a,c) := \left(\frac{\op{Re}(a)}{\sqrt{\op{Re}(a)^2 +  \op{Re}(c)^{2}}},\, \frac{\op{Re}(c)}{\sqrt{\op{Re}(a)^{2} +  \op{Re}(c)^{2}}}\right),\quad 0\neq(a,c)\in X_3,
\]
and define positive lifts $p$, $q\in A_{3}$ for $\tilde{p}$ and $\tilde{q}$ by
\[
p(a,c) :=  
\begin{cases}
(1-\op{Im}(a)^{2}) \cdot \tilde{p}(t(a,c))
&, \quad\text{if}\ \op{Im}(a)\neq1 ,\ \op{Im}(c)=0,	\\
(1-\op{Im}(c)^{2}) \cdot \tilde{p}(t(a,c))
&, \quad\text{if}\ \op{Im}(a)=0 , \ \op{Im}(c)\neq1 , \\
0 &, \quad\text{if}\ \operatorname{Im}(a)=1, \ \text{or}\ \op{Im}(c)=1,
\end{cases}
\]

\[
q(a,c) :=  
\begin{cases}
(1-\op{Im}(a)^{2}) \cdot \tilde{q}(t(a,c))
&, \quad\text{if}\ \op{Im}(a)\neq1 ,\  \op{Im}(c)=0, \\
(1-\op{Im}(c)^{2}) \cdot \tilde{q}(t(a,c))
&, \quad\text{if}\ \op{Im}(a)=0 , \ \op{Im}(c)\neq1,\\
0 &, \quad\text{if}\ \op{Im}(a)=1,\ \text{or}\ \op{Im}(c)=1.
\end{cases}
\]
Consequently, $\rho([\tilde{p}])=[ \, \exp(2\pi ip) \, ]$ and $\rho([\tilde{q}])=[ \, \exp(2\pi iq) \, ]$. Using the identification $K_1(I_1\oplus I_2)\cong K_1(I_1)\oplus K_1(I_2)$, we write
\[
e = (e_{1},e_{2}) := [ \exp(2\pi ip) ]\quad \text{and}\quad f = (f_{1},f_{2}) := [ \exp(2\pi iq) ].
\]
Let $(a,0)\in V_{2}$ be an arbitrary element, that is, $a\in \mathbb T$ with $-1<\op{Re}(a)<1$ and $\op{Im}(a)>0$. Recall that $\tilde{p}$ and $\tilde{q}$ satisfy 
\[
\tilde{p}(1,0) = \tilde{q}(-1,0)=\begin{pmatrix}1&0\\0&0\end{pmatrix}\quad \text{and} \quad \tilde{p}(-1,0) = \tilde{q}(1,0) = \begin{pmatrix} 0 & 0 \\ 0 & 1 \end{pmatrix}.
\]
Thus, if $r(a):=\exp(-2\pi i\op{Im}(a)^{2})$, then
\[
\exp(2\pi ip)(a,0) = 
\begin{cases}
\begin{pmatrix}1&0\\0&r(a)\end{pmatrix}&, \quad \text{if}\ \op{Re}(a)<0,\\
\begin{pmatrix}r(a)&0\\0&1\end{pmatrix}&, \quad \text{if}\ \op{Re}(a)>0,\\
\begin{pmatrix}1&0\\0&1\end{pmatrix}&, \quad \text{if}\ \op{Re}(a)=0,
\end{cases}
\]
and
\[
\exp(2\pi iq)(a,0) =  
\begin{cases}
\begin{pmatrix}r(a)&0\\0&1\end{pmatrix}&, \quad \text{if}\ \op{Re}(a)<0, \\
\begin{pmatrix}1&0\\0&r(a)\end{pmatrix}&, \quad \text{if}\ \op{Re}(a)>0,\\
\begin{pmatrix}1&0\\0&1\end{pmatrix}&, \quad \text{if}\ \op{Re}(a)=0.
\end{cases}
\]
Since $\overline{V_{2}}=V_{2}\cup \left\lbrace(1,0),\ (-1,0)\right\rbrace$ and $r(0)=1$, we can consider the restrictions of $\exp(2\pi ip)$ and $\exp(2\pi iq)$ to functions on $\overline{V_2}\subseteq X_3$ as elements in $C_{0}(V_{2},C^{*}(v_0))^{\sim}\cong C(\mathbb T,C^{*}(v_0))$. This identification can be chosen so that
\[
\exp(2\pi ip)\sim_{h} \op{diag}(z,\bar{z}),\ \exp(2\pi iq) \sim_{h} \op{diag}(\bar{z},z) \ \text{in} \ \mathcal U(C(\mathbb T,C^{*}(v_0))).
\]
A combination of Lemma \ref{I} and Remark \ref{I1I2} shows that $e_{2}=-f_{2}\in K_{1}(I_{2})$, and that $e_2$ is a generator for $K_1(I_2)\cong\Z$. Using the relations
\[
\tilde{p}(0,1) = \tilde{q}(0,1) = \frac{1}{2}\cdot \begin{pmatrix}1&1\\1&1\end{pmatrix}\quad \text{and}\quad \tilde{p}(0,-1) = \tilde{q}(0,-1) = \frac{1}{2}\cdot \begin{pmatrix}1&-1\\-1&1\end{pmatrix},
\]
one concludes that $\exp(2\pi ip)$ and $\exp(2\pi iq)$ coincide on $V_1$. A similar reasoning as above reveals that $e_{1}=f_{1}\in K_{1}(I_{1})$, and that $e_{1}$ is a generator for $K_{1}(I_{1})$. Altogether, we have shown that the image of $\rho$ is generated by $(1,1)$ and $(1,-1)\in \Z^2\cong K_1(I_1)\oplus K_1(I_2)$. Hence,
\[
K_0(A_3)\cong \op{ker}(\rho)\cong \Z\quad \text{and}\quad K_1(A_3)\cong \op{coker}(\rho)\cong \Z/2\Z.
\]
\end{proof}

The $K$-theory of $C(SU_{-1}(2))$ can now be computed by considering the exact sequence
\[
	\xymatrix{
	 0 \ar[r]& M_{2}(C_{0}(X_{4}\setminus X_{3})) \ar[r]^/0.6cm/\iota & A_{4} \ar[r]^{\pi_{3}} & A_{3} \ar[r] & 0.
	}
\]
Observe that $X_{4}\setminus X_{3}$ is homeomorphic to $\mathbb R^{3}$. By passing to the corresponding six-term exact sequence, we see with the help of Lemma \ref{A3} that $K_{0}(C(SU_{-1}(2)))=\mathbb Z\cdot [1]$, and $K_1(C(SU_{-1}(2)))$ fits into a short exact sequence
\[
	\xymatrix{
	 0 \ar[r]& \mathbb Z \ar[r] & K_{1}(C(SU_{-1}(2)))  \ar[r] & \Z/2\Z \ar[r] & 0.
	 }
\]
This implies that $K_{1}(C(SU_{-1}(2)))$ is isomorphic to either $\mathbb Z$ or $\mathbb Z\oplus\Z/2\Z$.

\begin{theorem} 
\label{K}
The $K$-theory for $C(SU_{-1}(2))$ satisfies
\[
K_{0}(C(SU_{-1}(2)))=\mathbb Z\cdot[1]\quad \text{and}\quad  K_{1}(C(SU_{-1}(2)))=\mathbb Z\cdot[u_{-1}].
\]
\end{theorem}
\begin{proof}
We only have to show that $K_{1}(C(SU_{-1}(2)))=\Z\cdot[u_{-1}]$. Consider the commutative diagram
\[
	\xymatrix{
	M_{2}(C_{0}(X_{4}\setminus X_{3})) \ar[r]^/0.6cm/{\iota} \ar@{-->}[d]_{\varphi}& A_4 \ar[d]^\cong \\
	M_{2}(C(SU(2))) & C(SU_{-1}(2)) \ar[l]_/-0.3cm/{\phi}\\
						}
\]
where $\phi$ is the injective $*$-homomorphism from Theorem \ref{fixed point}. Define unitaries $x,y \in \mathcal U_2(C(SU(2)))$ by 
\[
x(a,c) := \begin{pmatrix} a&\bar{c}\\c&-\bar{a}\end{pmatrix}\quad \text{and} \quad  y(a,c) := \begin{pmatrix} -a&\bar{c}\\c&\bar{a}\end{pmatrix}.
\]
Recall that $K_1(C(SU(2)))\cong \Z$ with generator $[u]$, where $u\in \mathcal U_2(C(SU(2)))$ is given by
\[
u(a,c):=\begin{pmatrix}a & -\bar{c} \\ c & a \end{pmatrix}.
\]
It is easily verified that
\[
K_{1}(\phi)([u_{-1}])=[x]+[y]=2\cdot [u]\in K_1(C(SU(2))).
\]
For $f\in C_{0}(X_{4}\setminus X_{3},\, M_{2}(\mathbb C))$, let $\tilde{f}\in C(SU(2),\, M_{2}(\mathbb C))$ denote the canonical extension, that is, $\tilde{f}(x)=0$ whenever $x\notin X_4\setminus X_3$. The identification $A_4\cong C(SU_{-1}(2))$ yields that
\[
\varphi(f) = \tilde{f}  +  \beta_r(\tilde{f}) + \beta_s(\tilde{f}) + \beta_{r}\circ\beta_s(\tilde{f}).
\]
In fact, $\varphi(f)\in C(SU(2),M_2(\mathbb C))$ is clearly $\beta$-invariant, and extends $f$. For the latter note that the four functions have disjoint open support. By the same reason, $\varphi$ is the sum of four $*$-homomorphism
\[
M_{2}(C_{0}(X_{4}\setminus X_{3})) \longrightarrow M_2(C(SU(2))).
\]
Moreover, all these $*$-homomorphism are homotopic, so that by additivity of $K$-theory, the image of $K_{1}(\varphi)$ is contained in $4\cdot K_{1}(C(SU(2)))$.

Now assume that $K_{1}(C(SU_{-1}(2)))\cong \mathbb Z\oplus\Z/2\Z$. In this case,
\[
K_1(\iota):K_1(C(X_4\setminus X_3))\longrightarrow K_1(C(SU_{-1}(2)))
\]
may be considered as the canonical embedding $\Z\to\Z\oplus\Z/2\Z$. Consequently, the image of $K_{1}(\varphi)$ coincides with the image of $K_{1}(\phi)$. In particular,
\[
2\cdot[u]\in 4\cdot K_1(C(SU(2))),
\]
which contradicts the fact that $[u]$ is a generator for $K_1(C(SU(2)))\cong \Z$. Thus, $K_{1}(C(SU_{-1}(2)))\cong \mathbb Z$.  It also implies that $K_{1}(\iota)$ is multiplication with $2$, and so $K_{1}(\phi)$ has to be multiplication with $2$, as well. Hence, $[u_{-1}]$ is a generator for $K_{1}(C(SU_{-1}(2)))$, which completes the proof.
\end{proof}

\section{\texorpdfstring{A continuous $C^*$-bundle over $[-1,0)$ with fibres  $C(SU_{q}(2))$}{Continuous bundle}}

In this section, we show the existence of a continuous $C^*$-bundle over $[-1,0)$ with the fibre at $q$ being isomorphic to $C(SU_{q}(2))$. The proof is very similar to Blanchard's analogous result for positive deformation parameters  \cite{Blanchard1996}, and makes use of the Haar state on $C(SU_q(2))$. 

Let $h:C(SU(2))\to \mathbb C$ be the state given by integration with respect to the Haar measure on $SU(2)$, and let $\op{tr}:M_2(\mathbb C)\to \mathbb C$ denote the normalised trace. It has been pointed out in \cite{Zakrzewski1991} that the Haar state $h_{-1}$ on $C(SU_{-1}(2))\subseteq C(SU(2))\otimes M_2(\mathbb C)$ coincides with the restriction of $h\otimes \op{tr}$. In particular, this shows that $h_{-1}$ is faithful, cf.\ \cite[I, §4, Lemma 1.8]{BrownOzawa2008}.  For the reader's convenience, we sketch a proof for the fact that $h\otimes\op{tr}$ restricts to $h_{-1}$.

\begin{proposition}
\label{HaarFaithfulq=-1}
The Haar state $h_{-1}$ is the unique state on $C(SU_{-1}(2))$ satisfying
\[
h_{-1}(\eta_{klm}) =
\begin{cases}
\frac{1}{m + 1} & ,\quad\text{if}\ k=0,\ l=m,\\
0 & ,\quad \text{otherwise,}
\end{cases}
\]
where $\eta_{klm}\in \mathcal B_{-1}$ is defined as in Section \ref{prelim}. Under the identification via the injective $*$-ho\-mo\-mor\-phism $\phi$ from Theorem \ref{fixed point}, $h_{-1}$ is given as the restriction of $h\otimes\op{tr}$. In particular, $h_{-1}$ is faithful.
\end{proposition}
\begin{proof}
Let $[\cdot,\cdot]$ denote the commutator on $C(SU_{-1}(2))\otimes C(SU_{-1}(2))$, and observe that
\[
[\alpha\otimes\alpha, \gamma^{*}\otimes \gamma ] = [ \gamma\otimes \alpha, \alpha^{*}\otimes \gamma ] = 0.
\]
Using this, one computes for $k$, $l$, and $m\in \mathbb N_{0}$ that
\[
\begin{array}{l}
\Delta_{-1}(\eta_{klm}) = \\
\sum\limits_{i=0}^{k}\tbinom k i \alpha^{i}{\gamma^{*}}^{k-i}\mathord\otimes
\alpha^{i}\gamma^{k-i}\sum\limits_{j=0}^{l}\tbinom l j \gamma^{j}{\alpha^{*}}^{l-j}\mathord\otimes\alpha^{j}\gamma^{l-j}\sum\limits_{p=0}^{m}\tbinom m p {\gamma^{*}}^{p}\alpha^{m-p}\mathord\otimes{\alpha^{*}}^{p}{\gamma^{*}}^{m-p}.
\end{array}
\]
Right invariance of the Haar state yields 
\begin{equation}
\label{RightInv}
\begin{split}
&h_{-1}(\eta_{klm}) \cdot 1 \ = \ \\
&\sum_{i=0}^{k} \sum_{j=0}^{l}\sum_{p=0}^{m}\tbinom k i\tbinom l j\tbinom m p
h_{-1}(\alpha^{i}{\gamma^{*}}^{k-i}\gamma^{j}{\alpha^{*}}^{l-j}{\gamma^{*}}^{p}\alpha^{m-p})\mathord\cdot
\alpha^{i}\gamma^{k-i}\alpha^{j}\gamma^{l-j}{\alpha^{*}}^{p}{\gamma^{*}}^{m-p},
\end{split}
\end{equation}
and similarly left invariance shows
\begin{equation}
\label{LeftInv}
\begin{split}
&h_{-1}(\eta_{klm})\cdot 1 \ = \ \\
&\sum_{i=0}^{k} \sum_{j=0}^{l}\sum_{p=0}^{m}\tbinom k i\tbinom l j\tbinom m p
h_{-1}(\alpha^{i}\gamma^{k-i}\alpha^{j}\gamma^{l-j}{\alpha^{*}}^{p}{\gamma^{*}}^{m-p})\mathord\cdot\alpha^{i}{\gamma^{*}}^{k-i}\gamma^{j}{\alpha^{*}}^{l-j}{\gamma^{*}}^{p}\alpha^{m-p}.
\end{split}
\end{equation}
Let $\pi_{(a,0)}$ be the character associated to $a\in \mathbb T$ from Section \ref{prelim}. Since $\pi_{(a,0)}(\gamma)=0$, all summands on the right hand side of \eqref{RightInv} vanish but the one associated to $i=k$, $j=l$ and $p=m$. Actually,
\[
h_{-1}(\eta_{klm}) = h_{-1}(\eta_{klm}) \cdot a^{k+l-m}.
\]
The analogous observation for \eqref{LeftInv} yields
\[
h_{-1}(\eta_{klm}) = h_{-1}(\eta_{klm}) \cdot a^{k-l+m}.
\]
Since these equations hold for every $a\in \mathbb T$, $h_{-1}(\eta_{klm})\neq0$ is only possible when
\[
k  + l - m = k - l + m = 0.
\]
This is equivalent to $k=0$ and $l=m$. In this case, by using the anticommutativity relations, \eqref{RightInv} simplifies to
\[
\begin{array}{lll}
h_{-1}(\eta_{0mm}) \cdot 1 & &\\
&\hspace{-2.1cm} = & \hspace{-2.1cm} \sum\limits_{p=0}^{m}{\tbinom m p}^{2} h_{-1}(\gamma^{p}{\gamma^{*}}^{p}\alpha^{m-p}{\alpha^{*}}^{m-p})\cdot \alpha^{p}{\alpha^{*}}^{p}\gamma^{m-p}{\gamma^{*}}^{m-p}\\
&\hspace{-2.1cm} = & \hspace{-2.1cm}  \sum\limits_{p=0}^{m}{\tbinom m p}^{2}
h_{-1}(\gamma^{p}{\gamma^{*}}^{p}(1-\gamma\gamma^{*})^{m-p})\cdot (1-\gamma\gamma^{*})^{p}\gamma^{m-p}{\gamma^{*}}^{m-p}\\
&\hspace{-2.1cm} = & \hspace{-2.1cm} \sum\limits_{p=0}^{m}\sum\limits_{i=0}^{m-p}\sum\limits_{j=0}^{p}{\tbinom m p}^{2}\tbinom {m-p} i \tbinom p j (-1)^{i+j} h_{-1}(\eta_{0(p+i)(p+i)})\cdot\eta_{0(m-p+j)(m-p+j)}.
\end{array}
\]
Assume that $m\geq1$, and let $c$ denote the coefficient of $\eta_{011}$ in the last expression (after reordering). Then
\[
c = -m\cdot h_{-1}(\eta_{0mm})  +  m^{2}\cdot h_{-1}(\eta_{0(m-1)(m-1)}) -  m^{2}\cdot h_{-1}(\eta_{0mm}).
\]
Since the family $\left\lbrace\eta_{0pp}\ :\ p\in \mathbb N_{0}\right\rbrace$ is linearly independent, $c=0$, and consequently
\[
h_{-1}(\eta_{0mm})=\frac{m^2}{m^2+m}\cdot h_{-1}(\eta_{0(m-1)(m-1)}).
\]
Hence, if we already know that
\[
h_{-1}(\eta_{0(m-1)(m-1)}) = \frac{1}{m},
\]
then
\[
h_{-1}(\eta_{0mm}) = \frac{m^{2}}{m^{2}+m}\cdot h_{-1}(\eta_{0(m-1)(m-1)}) = \frac{1}{m+1}.
\]
It is routine to verify that $h\otimes \op{tr}(\phi(\eta_{klm}))=h_{-1}(\eta_{klm})$.
\end{proof}

\begin{definition}[\cite{Kasparov1988}]
Let $X$ be a locally compact Hausdorff space. A \emph{$C_{0}(X)$-algebra} is a $C^{*}$-algebra $A$ together with a non-degenerate $*$-homo\-mor\-phism from $C_{0}(X)$ into the center of $\M(A)$.
\end{definition}
If $A$ is a $C_{0}(X)$-algebra, then for every $x\in X$, $C_{0}(X\setminus\left\lbrace x\right\rbrace)A$ forms a closed two-sided ideal. The respective quotient $A_{x}$ is called the \emph{fibre} at $x$. For $a\in A$, we denote by $a_{x}$ its image under the canonical surjection onto $A_{x}$. One fundamental property of $C_{0}(X)$-algebras is that the map
\[
X \longrightarrow \mathbb R, \quad x\mapsto \Vert a_{x}\Vert
\]
is upper-semicontinuous for every $a\in A$. We call $A$ a \emph{continuous $C^{*}$-bundle} over $X$ if all these maps are continuous.

\begin{theorem}
There exists a continuous $C^*$-bundle $A$ over $[-1,0)$ with the property that the fibre at $q\in[-1,0)$ is isomorphic to $C(SU_{q}(2))$.
\end{theorem}
\begin{proof}
Let $B$ be the universal unital $C^{*}$-algebra with generators $\alpha$, $\gamma$, and $f$, satisfying the following relations: $f$ is normal, has spectrum $[-1,0]$, commutes with $\alpha$ and $\gamma$, and
\[
\begin{pmatrix}\alpha& -f\gamma^{*}\\ \gamma & \alpha^{*}\end{pmatrix} \in M_{2}(B)
\]
is unitary. Since $f$ is central and normal, $B$ is a $C([-1,0])$-algebra. By the respective universal properties of $C(SU_{q}(2))$ and $B_{q}$, it follows that the natural homomorphism $C(SU_{q}(2))\to B_{q}$ is an isomorphism. Thus,
\[
A:= C_{0}([-1,0))B
\]
is a $C_{0}([-1,0))$-algebra with the property that the  fibre at $q\in[-1,0)$ is isomorphic to $C(SU_{q}(2))$. By Proposition \ref{HaarFaithful} and \ref{HaarFaithfulq=-1}, all Haar states are faithful, and we can apply \cite[Theorem 3.3]{Blanchard1996} stating that $A$ is a continuous $C^*$-bundle if there is a $C_{0}([-1,0))$-linear, positive map
\[
\varphi: A \longrightarrow C_{0}([-1,0))
\]
with $h_{q}=\op{ev}_{q}\circ \varphi$ for all $q\in [-1,0)$. A straightforward partition of unity argument shows that the $C_{0}([-1,0))$-linear span of
\[
\eta_{klm} := 
\begin{cases}
\alpha^{k}\gamma^{l}{\gamma^{*}}^{m}& ,\quad \text{if}\ k\geq0,\, m,n\geq0, \\
{\alpha^{*}}^{-k}\gamma^{l}{\gamma^{*}}^{m} & ,\quad \text{if}\ k<0,\, m,n\geq0,
\end{cases}
\]
is dense in $A$. For $a\in A$, we define a map 
\[
\varphi(a): [-1,0) \longrightarrow \mathbb C, \quad \varphi(a)(q) := h_{q}(a_{q}).
\]
Observe that $\varphi(g\cdot a)=g\cdot\varphi(a)$ for any $g\in C_0([-1,0))$. Moreover,
\[
\varphi(\eta_{0mm})(q) = \frac{1 - q^{2}}{1 - q^{2m+1}} \stackrel{q\longrightarrow-1}{\xrightarrow{\hspace*{1cm}}}\frac{1}{m+1} = \varphi(\eta_{0mm})(-1),
\]
and $\varphi(\eta_{klm})=0$ in all other cases. Hence, $\varphi(a)\in C_{0}([-1,0))$ for every $a\in A$, and $\varphi$ is a well-defined and $C_0([-1,0))$-linear map. Moreover, $\varphi$ is positive, since each $h_{q}$ is a state. The proof is complete.
\end{proof}

\bibliographystyle{plain}

\bibliography{D:/Mathe/Projekte/sbarlak}

\end{document}